\documentclass[10pt,a4paper]{article}
\usepackage[utf8]{inputenc}
\usepackage[T1]{fontenc}
\usepackage{amsmath}
\usepackage{amsthm}
\usepackage{amsfonts}
\usepackage{amssymb}
\usepackage{makeidx}
\usepackage{graphicx}
\usepackage{algorithm2e}
\usepackage{lmodern}
\usepackage{mathtools}
\usepackage{url}
\usepackage{xcolor}
\usepackage[left=2cm,right=2cm,top=3cm,bottom=3cm]{geometry}

\DeclareMathOperator{\seed}{seed}
\DeclareMathOperator{\mdim}{mdim}
\DeclareMathOperator{\EGE}{EGE}
\DeclareMathOperator{\gen}{gen}
\DeclareMathOperator{\wt}{wt}
\DeclareMathOperator{\Rs}{Sub}
\DeclareMathOperator{\supp}{supp}

\author{Prudence Djagba\thanks{AIMS Rwanda, Kigali. 
\href{mailto:prudence.djagba@aims.ac.rw}{prudence.djagba@aims.ac.rw}} 
\and Jan Hązła\thanks{AIMS Rwanda, Kigali. 
\href{mailto:jan.hazla@aims.ac.rw}{jan.hazla@aims.ac.rw}}}
\title{Combinatorics of subgroups of Beidleman near-vector spaces}
\date{}

\usepackage{hyperref}
\usepackage{cleveref}

\newtheorem{thm}{Theorem}
\newtheorem{lem}[thm]{Lemma}
\newtheorem{ques}[thm]{Question}
\newtheorem{cor}[thm]{Corollary}
\newtheorem{cl}[thm]{Claim}

\theoremstyle{remark}
\newtheorem{exa}[thm]{Example}
\newtheorem{rem}[thm]{Remark}

\theoremstyle{definition}
\newtheorem{defn}[thm]{Definition}

\begin{document}
\maketitle

\begin{abstract} 
Combinatorial aspects of $R$-subgroups of finite
dimensional Beidleman near-vector spaces over nearfields
are studied. A characterization of $R$-subgroups is used to obtain the smallest
possible size of a generating set of a subgroup, which is much smaller
than its dimension. Furthermore,
a formula for the number of $R$-subgroups of an $n$-dimensional near-vector space 
is developed.
\end{abstract}
\section{Introduction}

Nearfields are skewfields (also called division rings) that lack the distributive law on one side.
%A nearfield is called proper if it is not a skewfield.
They were first studied by Dickson \cite{dickson1905finite} in $1905$ and turned out to
have applications to geometry  and automata theory \cite{veblen1907, appl1980, automata}.
Almost all finite nearfields are constructed by distorting multiplication in finite fields
through a construction due to Dickson. There are only seven additional exceptional examples
\cite{zassenhauss1935}. 
%The smallest finite nearfields which are not fields
%have sizes $3^2,5^2$ and $4^3$.
For an overview of the subject, see books \cite{pilz2011near,meldrum1985near}.
Other works on nearfields 
include \cite{zassenhauss1935,ellerskarzel1964, wahling1987theorie,finitenearfield, djagba, djagbacenter}. 

In $1966$ the notion of near-vector spaces over nearfields was firstly defined by Beidleman \cite{beidleman1966near} 
%which extended  the concept of a vector space to a non-linear structure and 
using the concept of nearring modules and the left distributive law.
There is an alternative notion of near-vector spaces which was  defined by Andr\'{e} in $1974$ and used automorphisms in the construction, resulting in the right distributive law holding \cite{andre1974, Howell2008, sanon2020}. In this paper we focus on the former notion.

More recently, contributions
to the theory of Beidleman near-vector spaces have been made
in \cite{djagbathesis,djagbahowell18}.
In particular, a theory of subspaces and subgroups of near-vector spaces
was studied
%In \cite{djagba,djagbahowell18} the authors described the $R$-subgroups of finite dimensional Beidleman near-vector spaces and
and some notions like $R$-dimension, $R$-basis, seed set and seed number of an R-subgroup
were introduced.
Due to lack of distributivity, more anomalous behavior is possible for near-vector
spaces compared to vector spaces over fields. 

In particular, an $R$-subgroup of a near-vector space is its subset closed under vector
addition and vector-scalar multiplication. As usual, an $R$-subgroup generated
by a set of vectors is the smallest $R$-subgroup that contains all those vectors.
One of the results by Djagba and Howell~\cite{djagbahowell18} was to use
an explicit procedure
called ``Expanded Gaussian Elimination''
in order to characterize
$R$-subgroups generated by a finite set of vectors in a near-vector space.
It was observed as a consequence of that result that a near-vector space $R^n$
over a proper nearfield $R$ can be generated by strictly less than $n$ vectors.

In this paper we build on this observation. In particular, 
in \Cref{sec:mdim} we show that
$k$ vectors in a near-vector space over finite nearfield $R$ of size
$|R|=q$ can generate a space of exponentially
larger dimension $\frac{q^k-1}{q-1}$ and that this bound is tight.
In \Cref{sec:counting} we develop an exact formula
for the number of $R$-subgroups of a proper near-vector space. 
These numbers are generally exponentially smaller than in the case
of true vector spaces over finite fields.
Finally, in \Cref{sec:linearity-index} we pose a question about the complexity of the process of
generating a near-vector space from a given set of vectors.

\section{Preliminaries}\label{sec:prelim}

\subsection{Near fields and near-vector spaces}

Our exposition of the preliminaries
in this section largely follows~\cite{djagbahowell18}, see also
\cite{pilz2011near} for a more detailed presentation.

\begin{defn} %(\cite{meldrum1985near,pilz2011near})
    Let $(R, +, .)$ be a triple such that $(R, +)$ is a group, $(R, .)$ is a
semigroup, and $a.(b+c) = a.b+a.c$  for all $a, b, c \in R$. Then $(R, +, .)$ is
called a \emph{(left) nearring}.
\end{defn}

Let $(R, +, .)$  be a nearring. By distributivity it is easy to see that for all $r \in R$ we have $r . 0 = 0$. 
However,
it is not true in general that $0 .r = 0$ for all $r \in R$. Define $R_0 = \{r  \in R : 0 . r = 0\}$
to be the zero-symmetric part of $R$. A nearring is called \emph{zero-symmetric} if $R = R_0$ i.e.,
$0 . r = r . 0 = 0$ for all $r \in R$.

\begin{defn}[Nearfield]
Let $(R,+,.)$ be a zero-symmetric nearring.
If in addition $(R^*, .)$ is a group,
then $(R, +, .)$ is called a 
\emph{(left) nearfield}.

We will call a nearfield $R$ \emph{proper} if
it is not a skewfield, that is if there
exist $a,b,c\in R$ such that $(a+b).c\ne a.b+a.c$.
\end{defn}

It is known
that the additive group of a nearfield is abelian \cite{ellerskarzel1964}.
The assumption that $R$ is zero-symmetric is only needed to exclude a degenerate
case of $(\mathbb{Z}_2,+)$ with multiplication defined as $a.b=b$
(see Proposition~8.1 in~\cite{pilz2011near}). We also quickly point out that multiplication by $-1$ in a nearfield always commutes:

\begin{cl}[see~Corollary~2.10 in~\cite{HA22} or Proposition~8.10
in~\cite{pilz2011near}]
    Let $R$ be nearfield, then for every $r\in R$ it holds
    $(-1).x=x.(-1)=-x$.
\end{cl}

While in general distributivity on the right
$(\alpha+\beta).\gamma=\alpha.\gamma+\beta.\gamma$ fails to hold
in a nearfield, there is a useful special case in which it is true:
\begin{cl}\label{lem:2-zero-distributive}
Let $R$ be a nearfield, $\alpha,\beta,\gamma,x,y\in R$ such that
$\alpha.x+\beta.y=0$. Then,
$\alpha.x.\gamma+\beta.y.\gamma=(\alpha.x+\beta.y).\gamma=0$.
\end{cl}
\begin{proof}
Since $\alpha.x+\beta.y=0$, it follows
$\alpha.x=-\beta.y$ and $\alpha.x.\gamma=-\beta.y.\gamma$.
Consequently, $\alpha.x.\gamma+\beta.y.\gamma =0$.
\end{proof}

\begin{defn}[$R$-Module] %(\cite{beidleman1966near})
  A triple $(M, +, \circ) $ is called a \emph{(right) nearring module} over a (left)
nearring $R$ if $(M,+)$ is a group and $\circ: M \times R \to M$ such that $m\circ(r_1 +r_2) = m\circ r_1 +m\circ r_2$ and $m\circ(r_1.r_2) = (m\circ r_1)\circ r_2$ for all $r_1, r_2 \in R$ and $m \in M$.

We write $M_R$ to denote that $M$ is a (right) nearring module over a (left) nearring $R$.
\end{defn}

As is usual, from now on we will use $\cdot$ or  simply concatenation for both nearring multiplication
and vector-scalar multiplication. 
We also define the $R$-module $R^n$ (for some fixed $n\in \mathbb{N}$) with element-wise addition and element-wise scalar multiplication $R^n \times R \to R^n$ given by $(v_1,v_2,\ldots,v_n)\cdot r= (v_1r,v_2r,\ldots,v_nr)$. 

\begin{defn}[$R$-subgroup] %(\cite{beidleman1966near})
A subset $H$ of a nearring module $M_R$ is called an \emph{$R$-subgroup} if $H$
is a subgroup of $(M, +)$ and $HR = \{hr: h \in H, r  \in R\} \subseteq H.$
\end{defn}
\begin{defn}[Submodule] %(\cite{beidleman1966near})
Let $M_R$ be a nearring module. $N$ is a submodule of $M_R$ if $(N, +)$ is a normal subgroup of $(M, +)$, and $(m + n)r - mr  \in N$ for all $m \in M$, $n \in N$ and $r \in R$.
\end{defn}
\begin{defn}[Near-vector space] %(\cite{beidleman1966near})
Let $M_R$ be a nearring  and $R$ a nearfield. $M_R$ is called a 
\emph{(Beidleman) near-vector space} if $M_R$ is a nearring module which is a direct sum of  submodules that have no proper $R$-subgroups.
We say that a near-vector space is finite
dimensional if it is such a finite direct sum.
\end{defn}

We remark that there exists an alternative notion of \emph{Andr\'{e} near-vector spaces}
studied in the literature 
\cite{andre1974, Howell2008, sanon2020}. More recently, Djagba and Howell~\cite{djagbahowell18} added to the theory of Beidleman near-vector spaces.
Among others, they stated the following
classification.

\begin{defn}[$R$-module isomorphism]
    Let $M_R$ and $N_R$ be two modules. A function $\Phi:M\to N$ is a $R$-module isomorphism
    if it is a bijection that respects $\Phi(m+n)=\Phi(m)+\Phi(n)$ and $\Phi(mr)=\Phi(m)r$
    for every $m,n\in M$ and $r\in R$.
\end{defn}

\begin{thm}[\cite{djagbahowell18,beidleman1966near}] Let $R$ be a (left) nearfield and $M_R$ a (right) nearring module. $M_R$ is
a finite dimensional near-vector space if and only if $ M_R $ is isomorphic to $R^n$ for some positive integer 
$n$.
\end{thm}

In this paper we continue the study of the subgroup structure of 
finite-dimensional Beidleman near-vector spaces.
Accordingly, in the following we will restrict ourselves to the canonical case of $R^n$.

\begin{defn}[Weight and support]
Let $u\in R^n$. The \emph{support of $u$}
is defined as $\supp(u)=\{1\le i\le n:u_i\ne 0\}$. The \emph{weight of $u$}
is then given as $\wt(u)=|\supp(u)|$.
\end{defn}

\subsection{Subgroups of $R^n$}

%\subsection{EGE algorithm}
In \cite{djagbahowell18}, $R$-subgroups of finite dimensional
near vector spaces have been classified using
the Expanded Gaussian  Elimination ($\EGE$) algorithm.
This algorithm is used to construct the smallest $R$-subgroup containing
given finite set of vectors.
(Such $R$-subgroup exists since any intersection of subgroups is 
a subgroup.)

\begin{defn}%(\cite{djagbahowell18})
Let $V$ be a set of vectors. Define $\gen(V)$ to be the intersection of all $R$-subgroups containing $V$.
\end{defn}

Let $M_R$ be a nearring module. Let $V=\{v_1,\ldots, v_k\}$ be a set of vectors in $M_R$. Let $T$ be an $R$-subgroup of $M_R$.
\begin{defn}[Seed set]%(\cite{djagba})
We say that \emph{$V$ generates $T$} if $\gen(V)=T$.
In that case we say that $V$
is a \emph{seed set} of $T$. 
We also define the \emph{seed number} $\seed(T)$
to be the cardinality of a smallest seed set of $T$.
\end{defn}

In \cite{djagba} it was proved that each $R$-subgroup is
a direct sum of modules $u_iR$ of a special kind:

\begin{thm}[Theorem~5.12 in~\cite{djagbahowell18}]
\label{thm:ege}
Let $R$ be a proper nearfield
and $\{v_1,\ldots,v_k\}$ be vectors in $R^n$. Then,
$\gen(v_1,\ldots,v_k)=\bigoplus_{i=1}^\ell u_iR$,
where the $u_i$ are rows of some matrix $U=(u_{ij})\in R^{\ell\times n}$
such that each of its columns has at most one non-zero entry.
\end{thm}

In particular, basis vectors $u_1,\ldots,u_\ell$ from
\Cref{thm:ege} have mutually disjoint supports.
\Cref{thm:ege} was proved by analyzing an explicit procedure
termed ``Expanded Gaussian Elimination (EGE)''. This procedure
takes $v_1,\ldots,v_k$ and outputs $u_1,\ldots,u_\ell$.
Later on we will make use of the following corollary:

\begin{cor}\label{cor:subgroup-classification}
Let $R$ be a proper nearfield and $T\subseteq R^n$. Then,
$T$ is an $R$-subgroup if and only if $T=\bigoplus_{i=1}^\ell u_iR$
for some nonzero vectors $u_1,\ldots,u_\ell$ with mutually disjoint supports.
\end{cor}

\begin{proof}
If $T=\bigoplus_{i=1}^\ell u_iR$,
then it is straightforward to check that $T$ is an $R$-subgroup. 
(Note that if $u\in T$ and $r\in R$,
then $ur\in T$ due to the disjoint support property.)

Conversely, let $T\subseteq R^n$ be an $R$-subgroup. If $R$
is finite then the conclusion is obvious since we can apply
\Cref{thm:ege} to $\gen(T)=T$. Otherwise, consider the following procedure:
Let $T_0=\gen(\emptyset)=\{0\}$. For $k\ge 0$, as long as there exists
$v_{k+1}\in T\setminus T_k$, let $T_{k+1}=\gen(v_1,\ldots,v_{k+1})$.
Clearly, this results in a sequence of $R$-subgroups satisfying
$T_k\subseteq T_{k+1}\subseteq T$ for every $k$.

We argue that this sequence terminates with $T_k=T$ for some $k\le n$,
whereupon the decomposition of $T$ follows by \Cref{thm:ege}.
To see this consider $T_k$ such that the sequence has not terminated yet
and apply \Cref{thm:ege} to obtain $T_k=\bigoplus_{i=1}^\ell u_iR$
and $T_{k+1}=\bigoplus_{i=1}^{\ell'} u'_iR$. Since $T_k\subseteq T_{k+1}$,
the partition of $[n]$ formed by supports of $u'_1,\ldots,u'_\ell$ must
be a refinement of the partition formed by supports of $u_1,\ldots,u_\ell$
(in other words, for every $1\le i'\le\ell'$ there exists
$1\le i\le\ell$ such that $\supp(u'_{i'})\subseteq \supp(u_i)$,
and $\bigcup_i \supp(u_i)=\bigcup_{i} \supp(u'_i)$).
And since $T_{k+1}\setminus T_k\ne\emptyset$, this refinement must be strict,
in particular $\ell<\ell'\le n$. By tracking $\ell$ it follows that
the process terminates in at most $n$ steps.
\end{proof}

That leads to the notion of dimension of an $R$-subgroup:
\begin{defn}
Let $R=\bigoplus_{i=1}^\ell u_iR$ be an $R$-subgroup of $R^n$ for a proper nearfield 
$R$. Then, we say that \emph{$R$ has dimension $\ell$} and
write $\dim R=\ell$.
\end{defn}
Note that the dimension of any $R$-subgroup of $R^n$ is well-defined (unique) by a
similar argument as in the proof of \Cref{cor:subgroup-classification}.
Another corollary is that every $R$-subgroup of dimension $\ell$ is isomorphic 
to $R^\ell$ as a nearring:

\begin{cor} 
\label{thm:isomorphism}
Let $T=\bigoplus_{i=1}^\ell u_iR$ be an $R$-subgroup of $R^n$
and its decomposition into nonzero vectors with disjoint supports.

Define $\Phi:T\to R^\ell$ as follows. For $1\le i\le \ell$,
let $\Phi(u_ir_i)=e_i r_i$, where $e_i$ is the $i$-th standard
basis vector. For general $v=\sum_{i=1}^\ell v_i$, where $v_i\in u_iR$,
let $\Phi(v)=\sum_{i=1}^\ell \Phi(v_i)$.

Then, $\Phi$ is a nearring isomorphism. In particular, $S\subseteq T$
is a seed set of $T$ if and only if $\Phi(S)$ is a seed set
of $R^\ell$.

\end{cor}

\begin{proof}
The fact that $\Phi$ is a well defined isomorphism is easy to see: It is
a well-defined bijection by the direct sum property, $\Phi(v+v')=\Phi(v)+\Phi(v')$
follows by the left distributive law, and $\Phi(vr)=\Phi(v)r$ by
disjoint supports of $u_1,\ldots,u_\ell$.

It is also easy to check that $S$ generates $T$ if and only if
$\Phi(S)$ generates $R^\ell$.

\end{proof}

In his treatment of near-vector spaces \cite{beidleman1966near},
Beidleman defined basis and dimension for near-vector spaces
and proved their basic properties.
\Cref{thm:isomorphism} implies that each $R$-subgroup is in
fact a near-vector space. Our definition of dimension is consistent
with the original one due to Beidleman. We refer the reader
to \cite{beidleman1966near} for the details.

\begin{exa}

We will take the example of the near-vector space $R^5$ (where $R$ is the smallest
existing nearfield of order $3^2$,
see \cite{Howell2008}). % which has subgroups of dimensions $0,1,2,3$ and $4$.  
Let $S$ be an $R$-subgroup of the form $S= u_1 R\oplus u_2  R\oplus u_3  R$ where 
$u_1$, $u_2$, $u_3$ are the rows of matrix
%Note that $S$ have the matrix representation 
$\begin{pmatrix}
1&0&0&1&0\\0&1&1&0&0\\
0&0&0&0&1
\end{pmatrix}$.
Then, the isomorphism $\Phi$ is defined as
$ \sum_{i=1}^3 u_i r_i \in S \mapsto  (r_1,r_2,r_5) \in R^3$ 
%where $e_i$ is a standard basis vector of $R^3.$ 
One seed set of $S$ is $\{(1,0,0,1,1),(0,1,1,0,2)\}$,
and the corresponding seed set of $R^3$ is
$\{ \Phi((1,0,0,1,1)),\Phi((0,1,1,0,2) )\}=\{ (1,1,1),(0,1,2) \}$.

\end{exa}

In the following we will always assume that $R$ is a proper nearfield.
We will restrict ourselves to finite dimensional near-vector
spaces $R^n$, but we do not always assume that $R$ is finite.

\section{Maximum dimension generated by $k$ vectors}\label{sec:mdim}

In \cite{djagba}, the seed number of $R^n$ when $2 \leq n\leq 9$
was determined when $R$ is the finite nearfield of order $9$.
In this work we provide a general construction 
of smallest seed set for $R^n$, where $R$ is any finite proper nearfield.
To that end, we will determine the following number:

\begin{defn}
    Let $k \ge 1$ be a positive integer and $R$ a nearfield. Define
    \begin{align*}
\mdim(k,R)= \max \Biggl \{ n: \exists \; v_1, \ldots, v_k \in R^n,
\text{ such that }
\gen(v_1, \ldots, v_k)=R^n  \Biggl \} \;.
\end{align*}
\end{defn}
That is, $\mdim(k,R)$
is the maximum dimension of $R^n$ that can be generated by $k$ vectors.
%generated by the vectors $v_1, \ldots, v_k$ in $R^n$ for some $n.$
\begin{thm} For every finite proper nearfield $R$ and $n\ge 1$,
\begin{align*}
\mdim(k,R)&=\frac{|R|^k-1}{|R|-1}\;.
\end{align*}

\label{th}
\end{thm}

This result is in stark contrast to the case where $R$ is a field,
where of course $\mdim(k,R)=k$.
In order to prove \Cref{th}, we
start with a dual characterization of
$R$-subgroups. For that, we will need a notion of 
``scalar product'' in near-vector spaces.

\begin{defn}[Scalar product]
    Let $x,y \in R^n.$ Define the scalar product between $x$ and $y$ to be $\langle x,y \rangle= \sum_{i=1}^nx_iy_i.$

    Furthermore, for $T\subseteq R^n$ we define its \emph{orthogonal set}
    as $T^{\perp}=\{x\in R^n: \langle u,x\rangle=0 \; \forall u \in T\}$.
\end{defn}

While this will be convenient notation for us, this ``near scalar product''
satisfies hardly any of the usual properties.
For example, beware that an orthogonal set of $T$ does not need to be an $R$-subgroup,
even if $T$ itself is an $R$-subgroup.

\begin{defn}[Simple vector]
    Let $x \in R^n$. A vector $x$ is simple if it has weight one or two.
\end{defn}

\begin{thm}
Let $T$ be a subset of $R^n$. 
$T$ is an $R$-subgroup if and only if there exists a set of simple vectors $D$ such that $T=D^{\perp}$. 
\label{rr}
\end{thm}

\begin{proof}
Let us start with showing that if $T = D^\perp$ for some set of simple
vectors $D$, then $T$ is an $R$-subgroup. For that, we need to check
that $x,y\in T$, $r\in R$ imply $x+y\in T$, $xr\in T$. Accordingly,
let $e\in D$. We need to check that
$\langle e,x+y\rangle=\langle e,xr\rangle=0$ for every $x,y\in T$,
 and
 $r\in R$. We proceed in two cases depending on the weight
of $e$.

\textit{Case 1:} $\wt(e)=1$ i.e., $e$ can be presented as $e=(0, \ldots,0, \alpha, 0, \ldots,0)$ where $e_j=\alpha$ for some $1 \leq j \leq n$ and $\alpha \ne 0.$  We have
\begin{align*}
     \langle e,x +y \rangle= \alpha (x_j+y_j)= \alpha x_j + \alpha y_j=
     \langle e,x\rangle+\langle e,y\rangle=0.
\end{align*}
Also
\begin{align*}
     \langle e,xr \rangle= \alpha (x_jr)= (\alpha x_j)r =0.
\end{align*}

\textit{Case 2:} $\wt(e)=2$ i.e., $e$ can be written as $e=(0, \ldots,0, \alpha, 0, \ldots, \beta,0, \ldots, 0)$ where $e_j= \alpha \ne 0$ and $e_{j'}=\beta \ne 0$ for $j < j'$. We have
\begin{align*}
     \langle e,x +y \rangle  = \alpha (x_j+y_j) + \beta (x_{j'}+y_{j'})  
     = \alpha x_j + \beta x_{j'}
     +\alpha y_j + \beta y_{j'}=0.
\end{align*}
Finally,
\begin{align}\label{eq:04}
     \langle e,xr \rangle= \alpha x_jr+ \beta x_{j'}r =0
\end{align}
follows by \Cref{lem:2-zero-distributive}. Note that in the equation
~\eqref{eq:04} we used the fact that $e$ is a simple vector.
Since the analysis holds for every $e\in D$, indeed
$x+y,xr\in T$ and $T$ is an $R$-subgroup.

   Conversely, let's suppose that $T$ is an $R$-subgroup of $R^n$ of dimension $\ell$. By 
   \Cref{cor:subgroup-classification}, $T=\bigoplus_{i=1}^\ell u_iR$ where  
   $u_1,\ldots,u_\ell$ are nonzero vectors with mutually disjoint supports.
   Accordingly, let $J_i=\supp(u_i)$ and $J_0=[n]\setminus \bigcup_{i=1}^\ell J_i$.   
   
   Let us construct set $D$ as follows. First, for every $j\in J_0$
   put in $D$ vector $e$ of weight 1 such that $e_j=1$.
   For every $1\le i\le \ell$, let $J_i=\{j_1,\ldots,j_{a_i}\}$
   where $j_1<\ldots <j_{a_i}$.
   We put $a_i-1$ vectors
   of weight 2 into $D$ as follows: Let $2\le b\le a_i$,
   $\alpha=u_{i,j_1}$ and $\beta=u_{i,j_b}$.
   Then the $b$-th added vector $e$ satisfies 
   $e_{j_1} = 1$
   and $e_{j_b}=-\alpha\beta^{-1}$. All in all, we obtain a simple set
   $D$ and let $T'=D^{\perp}$.

   To complete the proof, we need to show $T=T'$.   
   To that end, let $u\in\{u_1,\ldots,u_\ell\}$
   and $e\in D$. By construction, either $\supp(e)\cap\supp(u)=\emptyset$
   or $\supp(e)\subseteq \supp(u)$. In the former case, clearly
    $\langle e,u\rangle=0$. Otherwise, recalling the notation $\alpha=u_{j_1}$
    and $\beta=u_{j_b}$ for some $1\le b\le j_i$, we see
    $\langle e,u\rangle=\alpha-\alpha\beta^{-1}\beta=0$.
   Accordingly, we established
   $\{u_1,\ldots,u_\ell\}\subseteq T'$. But in the first part of
   the proof we showed
   that $T'$ is an $R$-subgroup, so also $T\subseteq T'$.

   On the other hand let $v\in T'$. By construction of weight 1 vectors
   in $D$, for every $j\in J_0$ it holds $v_j=0$. Therefore, there is
   unique decomposition $v=\sum_{i=1}^\ell v_i$, where
   $\supp(v_i)\subseteq J_i$. Fix $1\le i\le \ell$ and
   recall that $J_i=\{j_1,\ldots,j_{a_i}\}$. By construction of vectors
   of weight two in $D$, for every $2\le b\le a_i$ we have
   \begin{align*}
       v_{i,j_1}-\alpha\beta^{-1} v_{i,j_b}
       =v_{i,j_1}-u_{i,j_1}u_{i,j_b}^{-1}v_{i,j_b}=0\;,
   \end{align*}
   which implies
   \begin{align*}
   v_{i,j_b}=u_{i,j_b} \cdot (u_{i,j_1}^{-1} v_{i,j_1})\;.
   \end{align*}
   Letting $r=u_{i,j_1}^{-1} v_{i,j_1}$ (which does not depend on $b$),
   we have $v_i=u_i r\in u_i R$.
   Accordingly, $v\in\sum_{i=1}^\ell u_i R$ and $T'\subseteq T$.
\end{proof}

\begin{defn}
     A vector $u$ is a \emph{left multiple} of $v$ if 
     there exists $r \in R$ such that $u_j=rv_j$ for
     every $1\le j\le n$.
\end{defn}

\begin{proof}[Proof of \Cref{th}]

Throughout the proof, we will make use of the following construction.
For $u \in R^k$, $u\ne 0$ define
\begin{align*}
b(u)= \min \Biggl\{ 1 \leq i \leq k: u_i \ne 0  \Biggl\}\;.
\end{align*}
Then, let
\begin{align*}
c(k, R)= \Biggl\{ u \in R^k: u \ne 0, u_{b(u)}=1 \Biggl\}
\end{align*}
Let us count the size of $c(k,R)$.
For fixed $b$, there are $|R|^{k-b}$  vectors $u$ such that $b(u)=b$.
%Let $ u \in c(k,R)$ such that $u=(0, \ldots, 0, 1, u_1,\ldots, u_m)$ where $b(u)+m=k$ and $u_i \in R.$  
So,
\begin{align*}
|c(k,R)|=\sum_{b=1}^k|R|^{k-b}=\sum_{\ell=0}^{k-1} |R|^{\ell}=\frac{|R|^k-1}{|R|-1}.
\end{align*}

Let $V=(v_{ij})$ be a matrix composed of columns $c(k,R)$ and denote the rows of $V$ by $v_1, \ldots, v_k$.
We know that $T=\gen(v_1, \ldots, v_k)$ is an $R$-subgroup of $R^n$ where $n=|c(k,R)|$. By Theorem \ref{rr}, there exists a set of simple vectors $D$
such that $T=D^{\perp}$.

We shall now show that for every simple $e \in R^n$ 
there exists $v\in\{v_1,\ldots,v_k\}$ such that $\langle e,v\rangle \ne 0.$ As before we proceed in two cases:

\textit{Case 1:} $\wt(e)=1$ i.e., $e$ can be presented as $e=(0, \ldots,0, \alpha, 0, \ldots,0)$ where $e_j=\alpha$ for some $1 \leq j \leq n$ and $\alpha \ne 0.$ By the construction of $c(k,R)$ every column of matrix $V$ is non-zero, in particular its $j^{th}$ column is also non-zero. Let $i$ be the row such that $v_{ij}=\beta \ne 0$. Then, $\langle e,v_i \rangle = \alpha \beta \ne 0.$

\textit{Case 2:} $\wt(e)=2$ i.e., $e$ can be presented as $e=(0, \ldots,0, \alpha, 0, \ldots, \alpha',0, \ldots, 0)$ where $e_j= \alpha \ne 0$ and $e_{j'}=\alpha' \ne 0$ for $j \ne j'$ (with other coordinates in $e$ to be equal to zero except the $j$ and $j'$-th coordinates). By construction of $c(k,R)$, for $j\ne j'$,
column $j$ of $V$ is never a left multiple of column $j'$. It follows that for every
$\lambda\in R$ there exists $1 \leq i \leq k$ such that $v_{ij}\ne \lambda v_{ij'}$. Choosing $ \lambda = - \alpha^{-1} \alpha '$ it follows $ \alpha v_{ij} + \alpha ' v_{ij'} \ne 0$, hence 
$\langle e,v_i\rangle \ne 0.$

To sum up, we showed that for every simple vector
$e\in R^n$ there exists $v\in\{v_1,\ldots,v_k\}\subseteq T$ such that
$\langle e,v\rangle\ne 0$. This implies $D=\emptyset$
and $T=R^n$. Accordingly, we showed that
$\mdim(k, R)\ge |c(k,R)|$.

We still need to justify $\mdim(k,R)\le |c(k,R)|$. We argue
this by contradiction. Let $V=(v_{ij})$ be any matrix with $k$ rows
and $n$ columns such that $n>|c(k,R)|$. Let $v_1,\ldots,v_k$
denote its rows and $c_1,\ldots,c_n$ its columns. We need
to argue that $\gen(v_1,\ldots,v_k)\ne R^n$.

To that end, notice that any matrix with $n>|c(k,R)|$ columns
must have two columns being left multiples of each other,
i.e., $c_{j'} = \lambda c_{j}$ for some $j' \ne j$ and $\lambda\in R$. Let $e$ be a vector such that $e_j=-\lambda$ and $e_{j'}=1$,
with other coordinates zero. Clearly, $e$ is a simple vector.
Let $1\le i\le k$. Then,
\begin{align*}
    \langle e,v_i\rangle
    =-\lambda v_{ij}+v_{ij'}=-\lambda v_{ij}+\lambda {v_{ij}}
    =0\;.
\end{align*}
Let $T=\{e\}^{\perp}$. By
\Cref{rr}, $T$ is an $R$-subgroup and we just showed that
$\{v_1,\ldots,v_k\}\subseteq T$. Hence,
$\gen(v_1,\ldots,v_k)\subseteq T\subsetneq R^n$,
which is the contradiction that we were looking for.
\end{proof}

Solving the formula for $\mdim(k, R)$ with regard to $k$,
%and taking the only one positive value, we have:
we have:

\begin{cor} \label{cor:seed-number}
Let $T$ be $R$-subgroup of $R^n$ such that $\dim(T)=l$. Then 

\begin{align*}
\seed(T)= \min \Biggl\{ k: \frac{|R|^k-1}{|R|-1} \geq l \Biggl\}
\;.
\end{align*}
In particular, for $\ell\ge 1$,
\begin{align*}
    \log_{|R|}\ell
    \le \seed(T)
    \le 2+\log_{|R|}\ell\;.
\end{align*}
\end{cor}

\begin{proof}
First, assume that $T=R^n$. Let $d(k,R)=\frac{|R|^k-1}{|R|-1}$.
By \Cref{th}, if $d(k,R)<n$, then $\seed(T)>k$. On the other hand,
let $k$ be such that $n\ge d(k,R)$. By \Cref{th},
there exist $v_1,\ldots,v_k\in R^{d(k,R)}$ such
that $\gen(v_1,\ldots,v_k)=R^{d(k,R)}$. Let
$v'_1,\ldots,v'_k\in R^n$ be $v_1,\ldots,v_k$ restricted to their
first $n$ coordinates. It follows, e.g.,
from \Cref{thm:gen-lc}, that
$\gen(v'_1,\ldots,v'_k)=R^n$ and therefore
$\seed(T)\le k$. Overall, $\seed(T)=\min\{k:d(k,R)\ge n\}$, as claimed.

For any other $R$-subgroup $T$ of dimension $\ell$ the formula
for $\seed(T)$ follows immediately by the reasoning above and \Cref{thm:isomorphism}.
The ``in particular'' statement follows by calculation.

\end{proof}

Using Theorem \ref{th}, we can derive an interesting version of this result
for infinite nearfields.
\begin{lem}
	Let $R$ be an infinite proper nearfield. There exists some vectors $v$ and $w$ such that $\gen(v,w)=R^n$ for any $n$.
\end{lem}
\begin{proof}
We follow the construction from the proof of \Cref{th}.
Since $R$ is infinite,
it is easy to see that $c(2,R)$ is infinite.
It is readily checked that any matrix with two rows and columns
forming a subset of $c(2,R)$ of size $n$ generates $R^n$.

\end{proof}

\begin{exa}
   \cite{zemmer1964} introduced an example of a proper infinite nearfield.
   Consider the ﬁeld of rational functions with rational coefficients $
   K=\mathbb{Q} (x)$.
   %of the integral domain of polynomials $(\mathcal{Q}, +, .) $ over the rational numbers. 
   Define a new multiplication in $K$ as follows
    \begin{align*}
    \frac{g(x)}{h(x)} \circ \frac{p(x)}{q(x)}=
         \begin{cases}
    \frac{g(x+d)}{h(x+d)} \cdot \frac{p(x)}{q(x)},  
    &\text{if }  \frac{p(x)}{q(x)} \ne 0 
    \\
   0, & \text{if }  \frac{p(x)}{q(x)}=0\\
  \end{cases}
    \end{align*}
    where $d= d(\frac{p(x)}{q(x)})= \deg (p(x))-\deg (q(x))$ is the degree of the fraction $\frac{p(x)}{q(x)}$.
    Then we obtain the inﬁnite near-ﬁeld $(K, +, \circ)$.
\end{exa}

\section{Counting $R$-subgroups}\label{sec:counting}

The main objective of this section is to count the number of $R$-subgroups
of $R^n$ of a given dimension for a proper finite nearfield $R$.

\begin{defn}
We will denote by $\Rs(R,\ell,n)$ the number of $R$-subgroups
of $R^n$ which have dimension $\ell$. Similarly, we will write 
$\Rs(R,n)$ for the number of all $R$-subgroups of $R^n$.
\end{defn}

\Cref{cor:seed-number} states that every $\ell$-dimensional
$R$-subgroup has a seed set of a size that is logarithmic in $\ell$.
Since the number of all subsets of $R^n$ of this size is not too large,
this immediately implies that there are not too many $R$-subgroups of dimension $\ell$:

\begin{lem}\label{lem:subgroup-count-simple}
Let $T$ be an $R$-subgroup of dimension $\ell$ and $k=\seed(T)$. Then,
$\Rs(R,\ell,n)\le\binom{|R|^n}{k}$.
In particular, for $\ell\ge 1$, in a setting where $\ell$ and $|R|$ are constant and $n$ goes
to infinity, it holds
\begin{align*}
    \Rs(R,\ell,n)\le \ell^n |R|^{2n}\exp(o(n))\;.
\end{align*}
\end{lem}

These bounds are in contrast to the case of vector spaces over finite fields. 
There, it is known that the
number of subspaces of $F^n$ of dimension $\ell$
is $|F|^{n\ell(1+o(1))}$.

\begin{proof}[Proof of \Cref{lem:subgroup-count-simple}]
The upper bound $\Rs(R,\ell,n)\le\binom{|R|^n}{k}$ follows since
every $R$-subgroup of dimension $\ell$ has at least one seed set of size $k$,
and there are $\binom{|R|^n}{k}$ of such subsets in $R^n$.
Furthermore, by \Cref{cor:seed-number} and standard bounds on binomial coefficients,
\begin{align*}
\Rs(R,\ell,n)&\le\binom{|R|^n}{k}\le (e|R|^n)^k\le
(e|R|^n)^{\log_{|R|}\ell+2}
=\ell^n|R|^{2n}\exp(o(n))\;.
\qedhere
\end{align*}
\end{proof}

In fact, we can use the structure of $R$-subgroups to 
give an exact formula for $\Rs(R,\ell,n)$:

\begin{thm}\label{thm:subgroup-count}
We have

  \begin{align}\label{eq:02}
        \Rs(R,\ell, n)=
        \sum_{d=0}^{n-\ell}\binom{n}{d}
        \begin{Bmatrix}n-d \\ 
        \ell\end{Bmatrix}
        \big(|R|-1\big)^{n-d-\ell}\;,
    \end{align}
    where $\begin{Bmatrix}n \\ k\end{Bmatrix}$ denotes the Stirling number of the 
    second kind.
\end{thm}

In order to prove \Cref{thm:subgroup-count}, we will apply
\Cref{cor:subgroup-classification}. But for a precise count we need
to understand how much double counting occurs in that result:

\begin{lem}\label{lem:double-counting}
Let $u_1,\ldots,u_\ell$ be a sequences of vectors
in $R^n$ with mutually disjoint supports
and let $u'_1,\ldots,u'_\ell$ be another such sequence.
Then, $\bigoplus_{i=1}^\ell u_iR=\bigoplus_{i=1}^\ell u'_iR$
if and only if there exists a permutation $\pi:[\ell]\to[\ell]$
and $r_1,\ldots,r_\ell\in R\setminus\{0\}$ such that
$u_i=u'_{\pi(i)}r_i$ for every $1\le i\le\ell$.
\end{lem}

\begin{proof}
In the easy direction, assume that $u_i=u'_{\pi(i)}r_i$ for every $1\le i\le\ell$.
Since $r_i\ne 0$, that implies $u_iR=u'_{\pi(i)}R$ and hence
\begin{align*}
\bigoplus_{i=1}^\ell u_iR=
\bigoplus_{i=1}^\ell u'_{\pi(i)}R=
\bigoplus_{i=1}^\ell u'_i R\;.
\end{align*}

Conversely, assume that there are no permutation $\pi$ and $r_1,\ldots,r_\ell$
satisfying the stated conditions. 
Let $T=\bigoplus_{i=1}^\ell u_iR$
and $T'=\bigoplus_{i=1}^\ell u'_iR$.
We proceed by considering two cases.

First, assume that
there exists $\pi$ such that $\supp(u_i)=\supp(u_{\pi(i)})$ for every $i$.
Then, it follows that there exists $i$ such that
$u_i\ne u_{\pi(i)}r$ for every $r\in R^*$ and consequently
$u_iR\ne u_{\pi(i)}R$. Since $\supp(u_i)=\supp(u'_{\pi(i)})$,
that means $u_i\in T\setminus T'$ and $T\ne T'$.

On the other hand, assume that there is no permutation $\pi$
such that $u_i=u'_{\pi(i)}$ for every $i$.
Then, there exist $i_1,i_2$ such that
$\supp(u_{i_1})\cap\supp(u'_{i_2})\ne\emptyset$
and $\supp(u_{i_1})\ne\supp(u'_{i_2})$. Let $u=u_{i_1}$
and $u'=u'_{i_2}$. Assume without loss of generality that there exist
$j'\in\supp(u')\setminus\supp(u)$ and $j\in\supp(u)\cap\supp(u')$.
Let $\alpha=u'_j$ and $\beta=u'_{j'}$.  Then, for every $v\in T'$
it holds $v_{j}=\alpha\beta^{-1}v_{j'}$.
On the other hand, for every $r,r'\in R$ there exists $v\in T$ such
that $v_j=r$ and $v_{j'}=r'$. Altogether again it follows that
$T\ne T'$.
\end{proof}

\begin{rem}
    Let $u_1,\ldots,u_\ell$ be a sequence of vectors with nonzero disjoint supports
    and let $U$ be a matrix with rows $u_1,\ldots,u_\ell$.
    The matrices formed this way are those that can be returned by the
    EGE algorithm from \cite{djagbahowell18}. By \Cref{lem:double-counting}, 
    each $R$-subgroup of dimension $\ell$ corresponds to $\ell!(|R|-1)^{\ell}$ such matrices.
\end{rem}

\begin{proof}[Proof of \Cref{thm:subgroup-count}]
Fix $J_0\subseteq [n]$ and let $d=|J_0|$. The number of nonzero sequences $u_1,\ldots,u_\ell$ with mutually
disjoint supports such that
$\bigcup_{i=1}^\ell \supp(u_i)=[n]\setminus J_0$ is equal to
$\begin{Bmatrix}n-d \\ 
\ell\end{Bmatrix}\ell! (|R|-1)^{n-d}$,
since there are $\begin{Bmatrix}n-d \\ 
\ell\end{Bmatrix}\ell!$ ways to choose the supports of $u_1,\ldots,u_\ell$
and $(|R|-1)^{n-d}$ ways to assign values of nonzero coordinates
for every fixed choice of supports. 

Summing over $J_0$, the total
number of vector sequences with mutually disjoint supports
is \begin{align*}
    \sum_{d=0}^{n-\ell}\binom{n}{d}\begin{Bmatrix}n-d \\ 
\ell\end{Bmatrix}\ell! (|R|-1)^{n-d}.
\end{align*}

By \Cref{lem:double-counting}, each subgroup is generated
by $\ell!(|R|-1)^{\ell}$ different sequences. Dividing this out, we obtain
\eqref{eq:02}.
\end{proof}

We have the following table for $\Rs(R,n)$ for some known nearfields:

\begin{center}
\begin{tabular}{ c c c c c c c c c c c }
 $\Rs(R,n)\text{ |}$  $n$ &  0 & 1 &2 &3 & 4& 5 &6 &7 \\

$R=3^2$ &  $1$ &  $2$& $12$  & $120$&$1424$ &$19488$ &$307904$ &$5539712$  \\

$R=4^3$ & $1$ & $2$& $67$  & $4355$ & $295234$ &$21036803$ &$1625419909$ &$140823067772$  \\

$R=5^4$ & $1$ &   $2$& $628$  & $393128$& $247268752$ &$156500388128$ &$100264147266880$ &$65739252669562496$  \\

\end{tabular}
\end{center}
These sequences appear to be known in OEIS (On-Line Encyclopedia of Integer Sequences) as "Dowling numbers", that is $\Rs(R,n)$ is the ``Dowling sequence with $b=|R|-1$'' see \cite{oeis}.

 The following table contains the  values of  $  \Rs(R,l, n)$ for $|R|=3^2$ and for $(l,n) \in \{0,1,\ldots, 8 \}^2$.

\begin{center}
\begin{tabular}{ c c c c c c c c c c c }
 $n$ | $k$ &  0 & 1 &2 &3 & 4& 5 &6 &7 &8 \\
$0$ &  $1$ &  & & & & & & &  \\
$1$ &  $1$ & $1$ & & & & & & &   \\
$2$ &  $1$& $10$  & $1$& & & & & &  \\
$3$ &  $1$& $91$ & $27$ & $1$ & & & & &   \\
$4$ &  $1$&  $820$ & $550$ & $52$& $1$ & & & &   \\
$5$ &  $1$& $7381$ &$10170$ & $1850$ & $85$ & $1$ & & &  \\ 

 $6$ & $ 1$& $66430$  & $180271$ & $56420$  & $4655$ & $126$  & $1$ & & \\
 
 $7$ &  $1$ & $597871$ & $3131037$  & $1590771$ & $210035$ & $9821$ &$ 175$ & $1$ &   \\
 $8$ &  $1$& $5380840$ & $53825500$ &$42900312$ & $8521926$ & $612696$ & $18396$ & $232$ & $1$  \\

 \end{tabular}
\end{center}

 The following table contains the  values of  $  \Rs(R,l, n)$ for $|R|=4^3$ and for $(l,n) \in \{0,1,\ldots, 7 \}^2$. 

\begin{center}
\begin{tabular}{ c c c c c c c c c c c }
 $n$ | $k$ &  0 & 1 &2 &3 & 4& 5 &6 &7  \\
$0$ &  $1$ &  & & & & & &  \\
$1$ &  $1$ & $1$ & & & & & &   \\

$2$ &  $1$& $65$  & $1$& & & & &   \\
$3$ &  $1$& $4161$ & $192$ & $1$ & & & &    \\
$4$ &  $1$&  $266305$ & $28545$ & $382$& $1$ & & &   \\
$5$ &  $1$& $17043521$ &$3891520$ & $101125$ & $635$ & $1$ & &  \\ 

 $6$ & $ 1$& $1090785345$  & $511266561$ & $23105270$  & $261780$ & $951$  & $1$ &  \\
 
 $7$ &  $1$ & $69810262081$ & $66021638592$  & $4901267861$ & $89335610$ & $562296$ &$ 1330$ & $1$    \\

 \end{tabular}
\end{center}

The following table contains the  values of  $  \Rs(R,l, n)$ for $|R|=5^4$ and for $(l,n) \in \{0,1,\ldots, 6 \}^2$. 

\begin{center}
\begin{tabular}{ c c c c c c c c c c c }
 $n$ | $k$ &  0 & 1 &2 &3 & 4& 5 &6 \\
$0$ &  $1$ &  & & & & &  \\
$1$ &  $1$ & $1$ & & & & &    \\

$2$ &  $1$& $626$  & $1$& & & &  \\
$3$ &  $1$& $391251$ & $1875$ & $1$ & & &   \\
$4$ &  $1$&  $244531876$ & $2733126$ & $3748$& $1$ & &    \\
$5$ &  $1$& $152832422501$ &$3658206250$ & $9753130$ & $6245$ & $1$ &  \\ 

 $6$ & $ 1$& $95520264063126$  & $4721932028751$ & $21925818740$  & $25346895$ & $9366$  & $1$  \\

 \end{tabular}
\end{center}

\section{Concluding comments}

In this article, we have shown a new construction of seed sets with the corresponding seed number. The novelty  is based on the construction of the set $c(k,R)$ that generate the maximal dimension of $R$-subgroups. Also we came up with a new characterization of $R$-subgroups of $R^n$ and use it in the counting technique of all possibles $R$-subgroups. 
We conclude with an open question about the complexity
of generating an $R$-subgroup from a set of vectors in $R^n$.

\label{sec:linearity-index}
Let $LC_0(v_1,v_2,\ldots,v_k):=\{ v_1,v_2,...,v_k\}$ and for $n\geq0$, let $LC_{n+1}$ be the set of all linear combinations of elements in $LC_n(v_1,v_2,\ldots,v_k)$, i.e.
\begin{equation*}
	LC_{n+1}(v_1,v_2,\ldots,v_k)=\left \{ \sum_{i=1}^\ell w_i\lambda_i
    \;|\; \ell\ge 0, w_i\in LC_n,\lambda_i\in R \;\forall 1\le i\le\ell \right\}.
\end{equation*}

\begin{thm}[Theorem 5.2 in~\cite{djagbahowell18}]
\label{thm:gen-lc} Let $v_1,v_2,\ldots,v_k \in R^n$. We have
\begin{equation*}
\gen(v_1,\ldots,v_k)=\bigcup_{i=0}^\infty LC_i(v_1,\ldots,v_k).
\end{equation*}
\end{thm}

It is an interesting question how many terms are really necessary in the sum
above. In other words, what is the smallest $i$ for which
$LC_i(v_1,\ldots,v_k)=LC_{i+1}(v_1,\ldots,v_k)$?
It is clear from \Cref{th}, as well as from the previous
work \cite{djagbahowell18} that sometimes $LC_1\ne LC_2$.

\begin{defn} \label{df3}
Let $v_1, \ldots, v_k \in R^n$. 
We define the index of $R$-linearity of $v_1, \ldots, v_k \in R^n$  to be
\begin{align*}
I(v_1, \ldots, v_k )= \min \{ p \in  \mathbb{N} \thickspace : \thickspace LC_p(v_1, \ldots, v_k )= \gen(v_1,\ldots,v_k)\}\;.
\end{align*}

\end{defn}

\begin{exa}
Taking $n=3$, it easy to check that for a proper
nearfield taking $v_1=(1,0,1)$ and $v_2=(1,1,0)$ in $R^3$ 
we have $\gen(v_1,v_2)=R^3$. In particular, $LC_2(v_1,v_2)=R^3\ne LC_1(v_1,v_2)$. Hence $I(v_1,v_2)=2$.
\end{exa}

Let $R$ be a finite nearfield of size $q$. 
A rough upper bound on the size of $LC_2(v_1,\ldots,v_k)$ is $q^{q^k}$.
Since
\begin{align*}
    q^k>\frac{q^k-1}{q-1}=\mdim(k,R)\;,
\end{align*}
this is more than the size of the largest possible near-vector space generated by
$k$ vectors, that is $q^{\mdim(k,R)}$.
This suggests the following open question:

\begin{ques}
    Does there exist any example of a near-vector space where
    $I(v_1,\ldots,v_k)> 2$ for some $v_1,\ldots,v_k$?
\end{ques}

More generally, can we find an  explicit expression, or at least some nontrivial bounds for 
 $I(v_1, \ldots$ $, v_k )$?

\paragraph{Acknowledgments}
We thank Dominic Bunnett for helpful
discussions. This work was supported by the Alexander von Humboldt Foundation German research chair
funding and the associated DAAD project No. 57610033
during a post-doctoral visit  fellowship at AIMS Rwanda. 

\bibliography{main}
\bibliographystyle{plain}
 
\end{document}